\documentclass{article}%
\usepackage{utopia}
\usepackage{amsmath}
\usepackage{amsfonts}
\usepackage{amssymb}
\usepackage{graphicx}
\usepackage[colorlinks=true, pdfstartview=FitV, linkcolor=blue, citecolor=red,
urlcolor=blue]{hyperref}%
\providecommand{\U}[1]{\protect\rule{.1in}{.1in}}
\newtheorem{theorem}{Theorem}[section]

\newtheorem{corollary}[theorem]{Corollary}

\newtheorem{lemma}[theorem]{Lemma}

\newtheorem{proposition}[theorem]{Proposition}
\newtheorem{remark}{Remark}

\newenvironment{proof}[1][Proof]{\noindent\textbf{#1.} }{\ \rule{0.5em}{0.5em}}
\numberwithin{equation}{section}
\allowdisplaybreaks
\begin{document}

\title{On Evaluations of Euler-type Sums of Hyperharmonic Numbers}
\author{Levent Karg\i n\thanks{lkargin@akdeniz.edu.tr}, M\"{u}m\"{u}n
Can\thanks{mcan@akdeniz.edu.tr}, Ayhan Dil\thanks{adil@akdeniz.edu.tr} and
Mehmet Cenkci \thanks{cenkci@akdeniz.edu.tr}\\Department of Mathematics, Akdeniz University, Antalya, Turkey}
\date{}
\maketitle

\begin{abstract}
We give explicit evaluations of the linear and non-linear Euler sums of
hyperharmonic numbers $h_{n}^{\left(  r\right)  }$ with reciprocal binomial
coefficients. These evaluations enable us to extend closed form formula of
Euler sums of hyperharmonic numbers to an arbitrary integer $r$. Moreover, we
reach at explicit formulas for the shifted Euler-type sums of harmonic and
hyperharmonic numbers. All the evaluations are provided in terms of the
Riemann zeta values, harmonic numbers and linear Euler sums.



\end{abstract}

\section{Introduction}

The classical linear Euler sum $\zeta_{H^{\left(  r\right)  }}\left(
p\right)  $ is the Dirichlet series
\begin{equation}
\zeta_{H^{\left(  r\right)  }}\left(  p\right)  :=\sum_{n=1}^{\infty}%
\frac{H_{n}^{\left(  r\right)  }}{n^{p}}, \label{g1}%
\end{equation}
where $H_{n}^{\left(  r\right)  }$ is the generalized harmonic number defined
by%
\[
H_{n}^{\left(  r\right)  }=\sum_{k=1}^{n}\frac{1}{k^{r}},\text{ }%
r\in\mathbb{N=}\left\{  1,2,3,\ldots\right\}  ,
\]
with $H_{n}^{\left(  1\right)  }=H_{n}$ and $H_{n}^{\left(  0\right)  }=n$.
When $r=1,$ $p=r$ and $p+r$\textbf{ }is odd, and for special pairs $\left(
p,r\right)  \in\{(2,4),(4,2)\}$, the sums of the form (\ref{g1}) have
representations in terms of the Riemann zeta values $\zeta\left(  r\right)  $
(see \cite{BBG,CB,LE,Ni}). In particular, the case $r=1$ yields to the
well-known Euler's identity \cite{LE,Ni}%
\begin{equation}
2\zeta_{H}\left(  p\right)  =\left(  p+2\right)  \zeta\left(  p+1\right)
-\sum_{j=1}^{p-2}\zeta\left(  p-j\right)  \zeta\left(  j+1\right)  ,\text{
}p\in\mathbb{N}\backslash\left\{  1\right\}  . \label{ES}%
\end{equation}

Many extensions of the Euler sums (so called Euler-type sums) involving
harmonic and generalized harmonic numbers have been studied extensively
(\cite{BBG,BK,CB,S2,S3,S4,S5,S,S1,SC,SS,Xu2017,XL,XW,XZZ,YW,ZX}). These
studies include the shifted Euler sums%
\[
\sum_{n=r+1}^{\infty}\frac{H_{n}}{\left(  n-r\right)  ^{p}},\text{ }\sum
_{n=1}^{\infty}\frac{H_{n}}{\left(  n+m\right)  ^{p}}\text{,}%
\]
and the linear and non-linear Euler sums with reciprocal binomial
coefficients
\[
\sum_{n=1}^{\infty}\frac{H_{n}^{\left(  r\right)  }}{n^{p}\binom{n+l}{l}%
},\text{ }\sum_{n=1}^{\infty}\frac{H_{n}^{\left(  r\right)  }H_{n}^{\left(
q\right)  }}{n^{p}\binom{n+l}{l}}.
\]
Recent studies also include hyperharmonic numbers with the connection of the
Dirichlet series%
\[
\zeta_{h^{\left(  r\right)  }}\left(  p\right)  :=\sum_{n=1}^{\infty}%
\frac{h_{n}^{\left(  r\right)  }}{n^{p}},\text{ }r\geq0\text{ and }p>r,
\]
which is called the Euler sums of hyperharmonic numbers. Here $h_{n}^{\left(
r\right)  }$ is the $n$th hyperharmonic number of order $r$ for $r\in
\mathbb{N}$, which is defined by \cite{CG}%
\[
h_{n}^{\left(  r\right)  }=\sum_{k=1}^{n}h_{k}^{(r-1)},\text{ }h_{n}^{\left(
1\right)  }=H_{n},
\]
and can be extended to negative order by \cite{DM}%
\begin{equation}
h_{n}^{\left(  -r\right)  }=\left\{
\begin{array}
[c]{lc}%
\frac{\left(  -1\right)  ^{r}}{\left(  n-r\right)  \binom{n}{r}}, &
n>r\geq1,\\
\sum\limits_{k=0}^{n-1}\binom{r}{k}\frac{\left(  -1\right)  ^{k}}{n-k}, &
r\geq n\geq1,
\end{array}
\right.  \label{1}%
\end{equation}
with the usual convention $h_{n}^{\left(  0\right)  }=1/n$. The Euler sums of
hyperharmonic numbers were first studied in \cite{MD} with some particular
values in terms of the Riemann zeta values. Later, Dil and Boyadzhiev
\cite{DB} extended Euler's identity (\ref{ES}) to the Euler sums of
hyperharmonic numbers\ as%
\begin{equation}
\zeta_{h^{\left(  r+1\right)  }}\left(  p\right)  =\frac{1}{r!}\sum_{k=0}^{r}%
\genfrac{[}{]}{0pt}{}{r+1}{k+1}%
\left\{  \zeta_{H}\left(  p-k\right)  -H_{r}\zeta\left(  p-k\right)
+\sum_{j=1}^{r}\mu\left(  p-k,j\right)  \right\}  , \label{HES}%
\end{equation}
where $%
\genfrac{[}{]}{0pt}{}{r}{k}%
$ is the Stirling number of the first kind and%
\begin{equation}
\mu\left(  p,j\right)  =\sum_{n=1}^{\infty}\frac{1}{n^{p}\left(  n+j\right)
}=\sum_{n=1}^{p-1}\frac{\left(  -1\right)  ^{n-1}}{j^{n}}\zeta\left(
p+1-n\right)  +\frac{\left(  -1\right)  ^{p-1}}{j^{p}}H_{j}. \label{Hp0j}%
\end{equation}
We remark that a slightly different form of (\ref{HES}) appears in \cite{KK}.
Besides, the series%
\[
\sum_{n=1}^{\infty}\frac{h_{n}^{\left(  r\right)  }}{n\binom{n+r}{n}},\text{
}\sum_{n=1}^{\infty}\frac{h_{n}^{\left(  r\right)  }}{\binom{n+r+1}{n}}\text{,
}\sum_{n=1}^{\infty}\frac{h_{n}^{\left(  r\right)  }}{\left(  n+m\right)
\binom{n+m+r}{r}},\text{ }m,r\in\mathbb{N}%
\]
are evaluated explicitly or represented as closed form formulas
(\cite{BL,CKDS,DB}).

One of the main theorems of this paper covers results on the foregoing series.

\begin{theorem}
\label{mteo1}For an integer $r$ and non-negative integers $l$, $m$ and $p$
with $p+l>r$, the linear Euler-type sum%
\[
\sum_{n=1}^{\infty}\frac{h_{n}^{\left(  r\right)  }}{\left(  n+m\right)
^{p}\binom{n+m+l}{l}}%
\]
can be written as a finite combination of the Riemann zeta values and harmonic numbers.
\end{theorem}

The proof depends on the evaluation of the series%
\[
\sum_{n=1}^{\infty}\frac{h_{n}^{\left(  r\right)  }}{n^{p}\binom{n+l}{l}%
},\text{ }\sum_{n=1}^{\infty}\frac{h_{n}^{\left(  -r\right)  }}{n^{p}%
\binom{n+l}{l}}%
\]
which we discuss them first. In particular, a perusal of the evaluation of the
second series reveals a closed form formula for the Euler sums of
hyperharmonic numbers of negative order: For $p,$ $r\in\mathbb{N}$,%
\begin{equation}
\zeta_{h^{\left(  -r\right)  }}\left(  p\right)  :=\sum_{n=1}^{\infty}%
\frac{h_{n}^{\left(  -r\right)  }}{n^{p}}=\zeta\left(  p+1\right)  +\sum
_{k=1}^{r}\left(  -1\right)  ^{k}\binom{r}{k}\left\{  \frac{H_{k}}{k^{p}}%
+\sum\limits_{j=2}^{p}\frac{H_{k}^{\left(  j\right)  }-\zeta\left(  j\right)
}{k^{p+1-j}}\right\}  . \label{teo1a1}%
\end{equation}
Thus (\ref{HES}) and (\ref{teo1a1}) provide closed form evaluations for the
Euler sums $\zeta_{h^{\left(  r\right)  }}\left(  p\right)  $, and hence of
the shifted Euler sums (Hurwitz-type Euler sums)
\[
\sum_{n=1}^{\infty}\frac{h_{n}^{\left(  r\right)  }}{\left(  n+m\right)  ^{p}%
}=\sum_{k=0}^{m}\binom{m}{k}\left(  -1\right)  ^{k}\zeta_{h^{\left(
r-k\right)  }}\left(  p\right)  \text{, }m,p\in\mathbb{N}\text{ and }p>r
\]
for arbitrary integer $r$.

Our second result, motivated from \cite{BaBG,FS,M,Xu2017,XL,XZZ,ZX,WL}, is on
the non-linear Euler sums of hyperharmonic numbers with reciprocal binomial coefficients.

\begin{theorem}
\label{mteo2}For an integer $q$ and non-negative integers $p$, $l$ and $r$
with $p+l>r+q$, the non-linear Euler-type sum%
\[
\sum_{n=1}^{\infty}\frac{h_{n}^{\left(  r\right)  }h_{n}^{\left(  q\right)  }%
}{n^{p}\binom{n+l}{l}}%
\]
can be written as a finite combination of the Riemann zeta values, harmonic
numbers and linear Euler sums.
\end{theorem}

In the task of proving Theorem \ref{mteo2} we further evaluate the series%

\[
\sum_{n=1}^{\infty}\frac{H_{n}}{n^{p}\left(  n+j\right)  \binom{n+l}{l}%
},\text{ }\sum_{n=1}^{\infty}\frac{h_{n}^{\left(  r\right)  }H_{n}}%
{n^{p}\binom{n+l}{l}},\text{ }\sum_{n=1}^{\infty}\frac{h_{n}^{\left(
r\right)  }}{n^{p}\left(  n+m\right)  \binom{n+l}{l}}.
\]

We finally focus our attention on the series $\sum\limits_{n=1}^{\infty}%
h_{n}^{\left(  r\right)  }/n^{p}\binom{n+l}{l}$ and particularly evaluate%
\[
\sum_{n=r+1}^{\infty}\frac{H_{n}}{\left(  n-r\right)  ^{p}\binom{n+q}{q}%
},\text{ }\sum_{n=q+r+1}^{\infty}\frac{\binom{n}{q}H_{n}}{\left(
n-r-q\right)  ^{p}},\text{ }\sum\limits_{k=1}^{\infty}\frac{\zeta\left(
p,k\right)  }{r+k},
\]
which are generalizations of the\textbf{ }shifted Euler sums of harmonic
numbers \cite[Theorem 2.1]{XL} and of the series involving the Hurwitz zeta
function $\zeta\left(  p,k\right)  $ \cite[p. 364]{DM}.

\section{Preliminary lemmas}

In this section we give some lemmas which we need in the sequel. The first
lemma is a direct consequence of the identity
\begin{align*}
\frac{1}{\left(  x+b\right)  ^{s}\left(  x+c\right)  ^{t}}  &  =\sum
\limits_{j=1}^{s}\frac{\left(  -1\right)  ^{s-j}}{\left(  c-b\right)
^{t+s-j}}\binom{s+t-j-1}{t-1}\frac{1}{\left(  x+b\right)  ^{j}}\\
&  +\sum\limits_{j=1}^{t}\frac{\left(  -1\right)  ^{t-j}}{\left(  b-c\right)
^{t+s-j}}\binom{s+t-j-1}{s-1}\frac{1}{\left(  x+c\right)  ^{j}},
\end{align*}
which can be deduced by the partial fraction decomposition.

\begin{lemma}
Let $N,$ $s,$ $t\in\mathbb{N}$. For non-negative integers $b$ and $c$ such
that $b\not =c,$ we have%
\begin{align}
\sum\limits_{n=1}^{N}\frac{1}{\left(  n+b\right)  ^{s}\left(  n+c\right)
^{t}}  &  =\sum\limits_{j=1}^{s}\frac{\left(  -1\right)  ^{s-j}}{\left(
c-b\right)  ^{t+s-j}}\binom{s+t-j-1}{t-1}\left(  H_{N+b}^{\left(  j\right)
}-H_{b}^{\left(  j\right)  }\right) \nonumber\\
&  +\sum\limits_{j=1}^{t}\frac{\left(  -1\right)  ^{t-j}}{\left(  b-c\right)
^{t+s-j}}\binom{s+t-j-1}{s-1}\left(  H_{N+c}^{\left(  j\right)  }%
-H_{c}^{\left(  j\right)  }\right)  . \label{5}%
\end{align}

\end{lemma}

The equation (\ref{5}) yields to the following lemma by letting $N\rightarrow
\infty.$

\begin{lemma}
Let $s,$ $t\in\mathbb{N}$. For non-negative integers $b$ and $c$ such that
$b\not =c,$ we have%
\begin{align}
\sum\limits_{n=1}^{\infty}\frac{1}{\left(  n+b\right)  ^{s}\left(  n+c\right)
^{t}}  &  =\sum\limits_{j=2}^{s}\frac{\left(  -1\right)  ^{s+j}}{\left(
c-b\right)  ^{t+s-j}}\binom{s+t-j-1}{t-1}\left(  \zeta\left(  j\right)
-H_{b}^{\left(  j\right)  }\right) \nonumber\\
&  +\sum\limits_{j=2}^{t}\frac{\left(  -1\right)  ^{s}}{\left(  c-b\right)
^{t+s-j}}\binom{s+t-j-1}{s-1}\left(  \zeta\left(  j\right)  -H_{c}^{\left(
j\right)  }\right) \nonumber\\
&  +\frac{\left(  -1\right)  ^{s}}{\left(  c-b\right)  ^{t+s-1}}%
\binom{s-1+t-1}{s-1}\left(  H_{b}-H_{c}\right)  . \label{3}%
\end{align}

\end{lemma}

For suitably selected sequences $\left\{  f_{n}\right\}  $, we remark that
\cite[p. 951]{XZZ}%
\begin{align}
\sum_{n=1}^{\infty}\frac{f_{n}}{n^{p}\left(  n+a\right)  }  &  =\sum
_{m=1}^{p-1}\frac{\left(  -1\right)  ^{m-1}}{a^{m}}\sum_{n=1}^{\infty}%
\frac{f_{n}}{n^{p+1-m}}+\frac{\left(  -1\right)  ^{p-1}}{a^{p-1}}\sum
_{n=1}^{\infty}\frac{f_{n}}{n\left(  n+a\right)  },\label{Zfpj0}\\
\sum_{n=1}^{\infty}\frac{f_{n}}{n^{p}\binom{n+l}{l}}  &  =\sum_{a=1}%
^{l}\left(  -1\right)  ^{a-1}\binom{l}{a}a\sum_{n=1}^{\infty}\frac{f_{n}%
}{n^{p}\left(  n+a\right)  }. \label{Zfp0l}%
\end{align}

The subsequent result serves as a combination of the equations above.

\begin{lemma}
Let $j,l,p\in\mathbb{N}$. Let $\left\{  f_{n}\right\}  $ be a sequence such
that the series $\sum_{n=1}^{\infty}\frac{f_{n}}{\left(  n+j\right)  \left(
n+s\right)  },$ $s\in\mathbb{N}$,\ is convergent. Then%
\begin{align}
&  \sum_{n=1}^{\infty}\frac{f_{n}}{n^{p}\left(  n+j\right)  \binom{n+l}{l}%
}\nonumber\\
&  =\sum_{m=1}^{p-1}\frac{\left(  -1\right)  ^{m-1}}{j^{m}}\left\{  \sum
_{n=1}^{\infty}\frac{f_{n}}{n^{p+1-m}}+\sum_{a=1}^{l}\left(  -1\right)
^{a}\binom{l}{a}\sum_{n=1}^{\infty}\frac{f_{n}}{n^{p-m}\left(  n+a\right)
}\right\} \nonumber\\
&  +\frac{\left(  -1\right)  ^{p-1}}{j^{p-1}}\sum_{s=0}^{l}\left(  -1\right)
^{s}\binom{l}{s}\sum_{n=1}^{\infty}\frac{f_{n}}{\left(  n+j\right)  \left(
n+s\right)  }. \label{Zfpjq}%
\end{align}

\end{lemma}

\begin{proof}
It can be seen that
\[
\sum_{n=1}^{\infty}\frac{f_{n}}{n^{p}\left(  n+j\right)  \binom{n+l}{l}}%
=\frac{1}{j}\sum_{n=1}^{\infty}\frac{f_{n}}{n^{p}\binom{n+l}{l}}-\frac{1}%
{j}\sum_{n=1}^{\infty}\frac{f_{n}}{n^{p-1}\left(  n+j\right)  \binom{n+l}{l}%
}.
\]
Employing this formula repetitively we find that
\begin{align*}
&  \sum_{n=1}^{\infty}\frac{f_{n}}{n^{p}\left(  n+j\right)  \binom{n+l}{l}}\\
&  =\sum_{m=1}^{p-1}\frac{\left(  -1\right)  ^{m-1}}{j^{m}}\sum_{n=1}^{\infty
}\frac{f_{n}}{n^{p-m+1}\binom{n+l}{l}}+\frac{\left(  -1\right)  ^{p-1}%
}{j^{p-1}}\sum_{n=1}^{\infty}\frac{f_{n}}{n\left(  n+j\right)  \binom{n+l}{l}%
}.
\end{align*}
By the partial fraction decomposition%
\begin{equation}
\frac{1}{\left(  x+k\right)  \left(  x+k+1\right)  \cdots\left(  x+l\right)
}=\sum_{s=k}^{l}\frac{\left(  -1\right)  ^{s-k}}{\left(  s-k\right)  !\left(
l-s\right)  !}\frac{1}{x+s}, \label{13}%
\end{equation}
we write the first series on the RHS as
\[
\sum_{n=1}^{\infty}\frac{f_{n}}{n^{p-m+1}\binom{n+l}{l}}=\sum_{a=0}^{l}\left(
-1\right)  ^{a}\binom{l}{a}\sum_{n=1}^{\infty}\frac{f_{n}}{n^{p-m}\left(
n+a\right)  },
\]
and the second as%
\[
\sum_{n=1}^{\infty}\frac{f_{n}}{n\left(  n+j\right)  \binom{n+l}{l}}%
=\sum_{s=0}^{l}\left(  -1\right)  ^{s}\binom{l}{s}\sum_{n=1}^{\infty}%
\frac{f_{n}}{\left(  n+j\right)  \left(  n+s\right)  },
\]
from which the proof follows.
\end{proof}

We conclude this section by the following lemma, which plays a critical role
in the proofs of the main theorems. It also provides extensions for
\cite[Proposition 6]{DB} and (\ref{HES}). Recall that the $r$-Stirling numbers
of the first kind are defined by \cite{Broder}
\begin{equation}
\left(  x+r\right)  \left(  x+r+1\right)  \cdots\left(  x+r+n-1\right)
=\sum_{k=0}^{n}%
\genfrac{[}{]}{0pt}{}{n}{k}%
_{r}x^{k}. \label{1s}%
\end{equation}
In particular, $%
\genfrac{[}{]}{0pt}{}{n}{k}%
_{0}=%
\genfrac{[}{]}{0pt}{}{n}{k}%
$ and $%
\genfrac{[}{]}{0pt}{}{n}{k}%
_{1}=%
\genfrac{[}{]}{0pt}{}{n+1}{k+1}%
$.

\begin{lemma}
\label{lemL10}Let $l,$ $p$\ and $r$ be non-negative integers with $p+l>r+1$.
Then the series
\[
\sum_{n=1}^{\infty}\frac{h_{n}^{\left(  r\right)  }}{n^{p}\binom{n+l}{l}}%
\]
can be written as a finite combination of the Riemann zeta values and harmonic numbers.
\end{lemma}

\begin{proof}
Multiplying both sides of \cite[p. 495]{DB}%
\begin{equation}
h_{n}^{\left(  r+1\right)  }=\frac{1}{r!}\sum_{k=0}^{r}%
\genfrac{[}{]}{0pt}{}{r+1}{k+1}%
n^{k}\left\{  H_{n}+\sum_{j=1}^{r}\frac{1}{n+j}-H_{r}\right\}  \label{6}%
\end{equation}
with $1/n^{p}\binom{n+l}{l}$ and then summing over $n$, we see that
\begin{align}
\sum_{n=1}^{\infty}\frac{h_{n}^{\left(  r+1\right)  }}{n^{p}\binom{n+l}{l}}
&  =\frac{1}{r!}\sum\limits_{j=0}^{r}%
\genfrac{[}{]}{0pt}{}{r+1}{j+1}%
\left\{  \sum_{n=1}^{\infty}\frac{H_{n}}{n^{p-j}\binom{n+l}{l}}\right.
\nonumber\\
&  \left.  +\sum_{v=1}^{r}\sum_{n=1}^{\infty}\frac{1}{n^{p-j}\left(
n+v\right)  \binom{n+l}{l}}-H_{r}\sum_{n=1}^{\infty}\frac{1}{n^{p-j}%
\binom{n+l}{l}}\right\}  . \label{L10}%
\end{align}
The proof is then completed when we write the series on the RHS of (\ref{L10})
as finite combinations of zeta values. The first series is \cite[Theorem
2]{SC}
\begin{align}
\sum_{n=1}^{\infty}\frac{H_{n}}{n^{p}\binom{n+l}{l}}  &  =\zeta_{H}\left(
p\right)  +\sum_{a=1}^{l}\binom{l}{a}\left(  -1\right)  ^{a}\left\{
\sum_{m=1}^{p-2}\frac{\left(  -1\right)  ^{m-1}}{a^{m}}\zeta_{H}\left(
p-m\right)  \right. \nonumber\\
&  \qquad\qquad\qquad\left.  +\frac{\left(  -1\right)  ^{p}}{2a^{p-1}}\left(
2\zeta\left(  2\right)  +\left(  H_{a-1}\right)  ^{2}+H_{a-1}^{\left(
2\right)  }\right)  \right\}  \label{L3}%
\end{align}
(which may also follow from (\ref{Zfp0l}) by taking $f_{n}=H_{n}$). The second
series is a consequence of (\ref{Zfpjq}) with $f_{n}=1$:
\begin{align}
&  \sum_{n=1}^{\infty}\frac{1}{n^{p}\left(  n+j\right)  \binom{n+l}{l}%
}\nonumber\\
&  =\sum_{m=1}^{p-1}\frac{\left(  -1\right)  ^{m-1}}{j^{m}}\left\{
\zeta\left(  p+1-m\right)  +\sum_{a=1}^{l}\binom{l}{a}\left(  -1\right)
^{a}\mu\left(  p-m,a\right)  \right\} \nonumber\\
&  +\frac{\left(  -1\right)  ^{p-1}}{j^{p-1}}\sum_{s=0}^{l}\left(  -1\right)
^{s}\binom{l}{s}B_{1}\left(  s,j\right)  . \label{Hplj}%
\end{align}
Here $\mu\left(  p,a\right)  $ is given by (\ref{Hp0j}) and
\[
B_{1}\left(  s,j\right)  =\left\{
\begin{array}
[c]{lc}%
\left(  H_{s}-H_{j}\right)  /\left(  s-j\right)  , & s\not =j,\\
\zeta\left(  2\right)  -H_{j}^{\left(  2\right)  }, & s=j.
\end{array}
\right.
\]
For the third series we take $f_{n}=1$ in (\ref{Zfp0l}) and see that%
\begin{equation}
\sum_{n=1}^{\infty}\frac{1}{n^{p}\binom{n+l}{l}}=\sum_{a=1}^{l}\binom{l}%
{a}\left\{  \sum_{m=1}^{p-1}\frac{\left(  -1\right)  ^{a+m}}{a^{m-1}}%
\zeta\left(  p+1-m\right)  +\frac{\left(  -1\right)  ^{a+p}}{a^{p-1}}%
H_{a}\right\}  . \label{L9}%
\end{equation}
These complete the proof.
\end{proof}

To see an example of how this lemma works, suppose that $r=l=2$ and $n=5$.
Then%
\begin{align*}
\sum_{n=1}^{\infty}\frac{h_{n}^{\left(  2\right)  }}{n^{5}\binom{n+2}{2}}  &
=-\frac{3}{2}\zeta\left(  5\right)  -\frac{1}{2}\left(  \zeta\left(  3\right)
\right)  ^{2}+\left(  \frac{5}{4}+\frac{1}{12}\pi^{2}\right)  \zeta\left(
3\right) \\
&  +\frac{1}{540}\pi^{6}-\frac{11}{1440}\pi^{4}-\frac{9}{32}\pi^{2}+\frac
{15}{8}.
\end{align*}

\section{Proof of theorems}

\subsection{Proof of Theorem \ref{mteo1}}

We start by recalling the binomial transform \cite[p. 43 Eq. (2)]{Ri}
\begin{equation}
a_{n}=\sum_{k=0}^{n}\binom{m}{n-k}b_{k}\Leftrightarrow b_{n}=\sum_{k=0}%
^{n}\binom{-m}{n-k}a_{k}, \label{dbinom}%
\end{equation}
and the upper negation identity%
\[
\binom{-m}{n-k}=\left(  -1\right)  ^{n-k}\binom{m-1+n-k}{n-k}.
\]
Applying the upper negation identity to the identity of hyperharmonic numbers
\cite{Benjamin}%
\[
h_{n}^{\left(  r+m\right)  }=\sum_{k=0}^{n}\binom{m-1+n-k}{n-k}h_{k}^{\left(
r\right)  }%
\]
we obtain
\[
\left(  -1\right)  ^{n}h_{n}^{\left(  r+m\right)  }=\sum_{k=0}^{n}\binom
{-m}{n-k}\left(  -1\right)  ^{k}h_{k}^{\left(  r\right)  }.
\]
Now (\ref{dbinom}) together with $b_{n}=\left(  -1\right)  ^{n}h_{n}^{\left(
r\right)  }$ and $a_{n}=\left(  -1\right)  ^{n}h_{n}^{\left(  r-m\right)  }$
yields%
\begin{equation}
h_{n}^{\left(  r-m\right)  }=\sum_{k=0}^{n}\binom{m}{k}\left(  -1\right)
^{k}h_{n-k}^{\left(  r\right)  }. \label{25}%
\end{equation}
Multiplying both sides of (\ref{25}) with $1/n^{p}\binom{n+l}{l}$ and summing
over $n$ give
\[
\sum_{n=1}^{\infty}\frac{h_{n}^{\left(  r-m\right)  }}{n^{p}\binom{n+l}{l}%
}=\sum_{k=0}^{m}\binom{m}{k}\left(  -1\right)  ^{k}\sum_{n=1}^{\infty}%
\frac{h_{n}^{\left(  r\right)  }}{\left(  n+k\right)  ^{p}\binom{n+k+l}{l}}.
\]
With the use of the classical binomial transform \cite[p. 43 Eq. (1)]{Ri}
\[
b_{m}=\sum_{k=0}^{m}\binom{m}{k}\left(  -1\right)  ^{k}a_{k}\Leftrightarrow
a_{m}=\sum_{k=0}^{m}\binom{m}{k}\left(  -1\right)  ^{k}b_{k}%
\]
we deduce that%
\begin{equation}
\sum_{n=1}^{\infty}\frac{h_{n}^{\left(  r\right)  }}{\left(  n+m\right)
^{p}\binom{n+m+l}{l}}=\sum_{k=0}^{m}\binom{m}{k}\left(  -1\right)  ^{k}%
\sum_{n=1}^{\infty}\frac{h_{n}^{\left(  r-k\right)  }}{n^{p}\binom{n+l}{l}}.
\label{23}%
\end{equation}
Lemma \ref{lemL10} then verifies the statement of Theorem \ref{mteo1}
for$\ r\geq m$.

The order of $h_{n}^{\left(  r-k\right)  }$ is negative for $r<m$, thus in
order to complete the proof, it is required to show that the following
series\textbf{ }%
\[
\sum_{n=1}^{\infty}\frac{h_{n}^{\left(  -q\right)  }}{n^{p}\binom{n+l}{l}%
},\text{ }q\geq1
\]
can be evaluated in terms of the Riemann zeta values and harmonic numbers.

\begin{theorem}
\label{teo1a}For $p,$ $r\in\mathbb{N}$ and non-negative integer $l$, we have%
\begin{align*}
&  \sum_{n=1}^{\infty}\frac{h_{n}^{\left(  -r\right)  }}{n^{p}\binom{n+l}{l}%
}=\zeta\left(  p+1\right)  -H_{r}^{\left(  p+1\right)  }+\sum_{n=1}^{r}%
\frac{1}{n^{p+1}\binom{n+l}{l}}+\sum_{k=1}^{r-1}\sum_{n=1}^{r-k}\binom{r}%
{k}\frac{\left(  -1\right)  ^{k}}{n\left(  n+k\right)  ^{p}\binom{n+k+l}{l}}\\
&  +\frac{\left(  -1\right)  ^{r}}{\binom{r+l}{l}}\sum_{\substack{a=0\\a\neq
r}}^{r+l}\binom{r+l}{a}\left(  -1\right)  ^{a}\left\{  \frac{H_{r}-H_{a}%
}{\left(  r-a\right)  ^{p}}+\sum_{j=2}^{p}\frac{H_{r}^{\left(  j\right)
}-\zeta\left(  j\right)  }{\left(  r-a\right)  ^{p-j+1}}\right\}  .
\end{align*}

\end{theorem}

\begin{proof}
From (\ref{1}) we have
\begin{align}
\sum_{n=1}^{\infty}\frac{h_{n}^{\left(  -r\right)  }}{n^{p}\binom{n+l}{l}}  &
=\sum_{n=1}^{r}\frac{h_{n}^{\left(  -r\right)  }}{n^{p}\binom{n+l}{l}}%
+\sum_{n=r+1}^{\infty}\frac{h_{n}^{\left(  -r\right)  }}{n^{p}\binom{n+l}{l}%
}\nonumber\\
&  =\sum_{n=1}^{r}\frac{1}{n^{p+1}\binom{n+l}{l}}+\sum_{k=1}^{r-1}\binom{r}%
{k}\left(  -1\right)  ^{k}\sum_{n=k+1}^{r}\frac{1}{\left(  n-k\right)
n^{p}\binom{n+l}{l}}\nonumber\\
&  +\sum_{n=r+1}^{\infty}\frac{\left(  -1\right)  ^{r}}{\left(  n-r\right)
n^{p}\binom{n+l}{l}\binom{n}{r}}. \label{2}%
\end{align}
The infinite series can be eqivalently written as
\[
\sum_{n=r+1}^{\infty}\frac{\left(  -1\right)  ^{r}}{\left(  n-r\right)
n^{p}\binom{n+l}{l}\binom{n}{r}}=\left(  -1\right)  ^{r}l!r!\sum_{n=1}%
^{\infty}\frac{1}{\left(  n+r\right)  ^{p}n\left(  n+1\right)  \cdots\left(
n+r+l\right)  }.
\]
Using (\ref{13}) gives
\begin{align*}
&  \sum_{n=r+1}^{\infty}\frac{\left(  -1\right)  ^{r}}{\left(  n-r\right)
n^{p}\binom{n+l}{l}\binom{n}{r}}\\
&  =\zeta\left(  p+1\right)  -H_{r}^{\left(  p+1\right)  }+\frac{\left(
-1\right)  ^{r}}{\binom{r+l}{l}}\sum_{\substack{a=0\\a\neq r}}^{r+l}%
\binom{r+l}{a}\left(  -1\right)  ^{a}\sum_{n=1}^{\infty}\frac{1}{\left(
n+r\right)  ^{p}\left(  n+a\right)  }.
\end{align*}
Then, from (\ref{3}), we obtain
\begin{align*}
&  \sum_{n=r+1}^{\infty}\frac{\left(  -1\right)  ^{r}}{\left(  n-r\right)
n^{p}\binom{n+l}{l}\binom{n}{r}}=\zeta\left(  p+1\right)  -H_{r}^{\left(
p+1\right)  }\\
&  +\frac{\left(  -1\right)  ^{r}}{\binom{r+l}{l}}\sum_{\substack{a=0\\a\neq
r}}^{r+l}\binom{r+l}{a}\left(  -1\right)  ^{a}\left\{  \sum_{j=1}^{p}%
\frac{H_{r}^{\left(  j\right)  }}{\left(  r-a\right)  ^{1+p-j}}-\sum_{n=2}%
^{p}\frac{\zeta\left(  n\right)  }{\left(  r-a\right)  ^{p+1-n}}-\frac{H_{a}%
}{\left(  r-a\right)  ^{p}}\right\}  ,
\end{align*}
which completes the proof.
\end{proof}

It is to be noted that using (\ref{13}) and (\ref{5}) the finite sums on the
RHS of (\ref{2}) can be written as%
\begin{align*}
\sum_{n=1}^{r}\frac{1}{n^{p+1}\binom{n+l}{l}}  &  =H_{r}^{\left(  p+1\right)
}+\sum_{a=1}^{l}\left(  -1\right)  ^{a}\binom{l}{a}\\
&  \qquad\qquad\times\left\{  \sum_{j=1}^{p}\frac{\left(  -1\right)  ^{p+j}%
}{a^{p+1-j}}H_{r}^{\left(  j\right)  }+\frac{\left(  -1\right)  ^{p}}{a^{p}%
}\left(  H_{a+r}-H_{a}\right)  \right\}
\end{align*}
and%
\begin{align*}
\sum_{n=k+1}^{r}\frac{1}{\left(  n-k\right)  n^{p}\binom{n+l}{l}}  &
=\sum_{a=1}^{l}\binom{l}{a}\frac{\left(  -1\right)  ^{a-1}a}{\left(
a+k\right)  }\left\{  \frac{H_{r-k}}{k^{p}}-\sum\limits_{j=1}^{p}\frac
{H_{r}^{\left(  j\right)  }-H_{k}^{\left(  j\right)  }}{k^{p+1-j}}\right. \\
&  \left.  -\sum\limits_{j=1}^{p}\left(  -1\right)  ^{p+j}\frac{H_{r}^{\left(
j\right)  }-H_{k}^{\left(  j\right)  }}{a^{p+1-j}}-\frac{\left(  -1\right)
^{p}}{a^{p}}\left(  H_{r+a}-H_{a+k}\right)  \right\}  .
\end{align*}
Letting $l=0$ in Theorem \ref{teo1a} and above formulas, we reach at
(\ref{teo1a1}), the closed form formula for the Euler sums of hyperharmonic
numbers of negative order.

\begin{remark}
Utilizing (\ref{23}), Lemma \ref{lemL10} and Theorem \ref{teo1a}, we may
present an illustrative example of Theorem \ref{mteo1} as follows:
\begin{align*}
\sum_{n=1}^{\infty}\frac{h_{n}^{\left(  2\right)  }}{\left(  n+4\right)
^{5}\binom{n+6}{2}}  &  =\frac{19}{2}\zeta\left(  5\right)  +\frac{3}%
{2}\left(  \zeta\left(  3\right)  \right)  ^{2}+\left(  \frac{15}{2}-\frac
{11}{12}\pi^{2}\right)  \zeta\left(  3\right) \\
&  -\frac{1}{420}\pi^{6}-\frac{43}{1440}\pi^{4}-\frac{7}{24}\pi^{2}-\frac
{533}{256}.
\end{align*}

\end{remark}

\subsection{Proof of Theorem \ref{mteo2}}

Similar to the verification of (\ref{L10}), we obtain%
\begin{align}
\sum_{n=1}^{\infty}\frac{h_{n}^{\left(  r+1\right)  }h_{n}^{\left(
q+1\right)  }}{n^{p}\binom{n+l}{l}}  &  =\frac{1}{r!}\sum\limits_{j=0}^{r}%
\genfrac{[}{]}{0pt}{}{r+1}{j+1}%
\left\{  \sum_{n=1}^{\infty}\frac{h_{n}^{\left(  q+1\right)  }H_{n}}%
{n^{p-j}\binom{n+l}{l}}\right. \nonumber\\
&  \quad\left.  +\sum_{v=1}^{r}\sum_{n=1}^{\infty}\frac{h_{n}^{\left(
q+1\right)  }}{n^{p-j}\binom{n+l}{l}\left(  n+v\right)  }-H_{r}\sum
_{n=1}^{\infty}\frac{h_{n}^{\left(  q+1\right)  }}{n^{p-j}\binom{n+l}{l}%
}\right\}  . \label{hnrhnq}%
\end{align}
Hence the series on the RHS of (\ref{hnrhnq}) is required to be evaluated.
Third series has been already evaluated in Lemma \ref{lemL10}. The following
proposition is about to calculation of the first series.

\begin{proposition}
\label{pro1}Let $l,p,r$ be non-negative integers with $p+l>r$. Then the series%
\[
\sum_{n=1}^{\infty}\frac{h_{n}^{\left(  r\right)  }H_{n}}{n^{p}\binom{n+l}{l}}%
\]
can be written as a finite combination of the Riemann zeta values, harmonic
numbers and the linear Euler sums.
\end{proposition}

\begin{proof}
We have%
\begin{align}
\sum_{n=1}^{\infty}\frac{h_{n}^{\left(  r+1\right)  }H_{n}}{n^{p}\binom
{n+l}{l}}  &  =\frac{1}{r!}\sum\limits_{j=0}^{r}%
\genfrac{[}{]}{0pt}{}{r+1}{j+1}%
\left\{  \sum_{v=1}^{r}\sum_{n=1}^{\infty}\frac{H_{n}}{n^{p-j}\binom{n+l}%
{l}\left(  n+v\right)  }\right. \nonumber\\
&  \quad\left.  +\sum_{n=1}^{\infty}\frac{\left(  H_{n}\right)  ^{2}}%
{n^{p-j}\binom{n+l}{l}}-H_{r}\sum_{n=1}^{\infty}\frac{H_{n}}{n^{p-j}%
\binom{n+l}{l}}\right\}  . \label{7}%
\end{align}
Now we deal with the series on the RHS of (\ref{7}). For the first series, we
set $f_{n}=H_{n}$ in (\ref{Zfpjq}) and deduce that
\begin{align}
&  \sum_{n=1}^{\infty}\frac{H_{n}}{n^{p}\left(  n+j\right)  \binom{n+l}{l}%
}\nonumber\\
&  =\sum_{m=1}^{p-1}\frac{\left(  -1\right)  ^{m-1}}{j^{m}}\left\{  \zeta
_{H}\left(  p+1-m\right)  +\sum_{a=1}^{l}\binom{l}{a}\left(  -1\right)
^{a}\sum_{n=1}^{\infty}\frac{H_{n}}{n^{p-m}\left(  n+a\right)  }\right\}
\nonumber\\
&  +\frac{\left(  -1\right)  ^{p-1}}{j^{p-1}}\sum_{s=0}^{l}\left(  -1\right)
^{s}\binom{l}{s}B_{2}\left(  s,j\right)  , \label{9}%
\end{align}
where
\[
B_{2}\left(  s,j\right)  =\left\{
\begin{array}
[c]{lc}%
\frac{1}{2j}\left(  2\zeta\left(  2\right)  +\left(  H_{j-1}\right)
^{2}+H_{j-1}^{\left(  2\right)  }\right)  , & s=0,\\
\frac{1}{2\left(  j-s\right)  }\left(  \left(  H_{j-1}\right)  ^{2}%
+H_{j-1}^{\left(  2\right)  }-\left(  H_{s-1}\right)  ^{2}-H_{s-1}^{\left(
2\right)  }\right)  , & s\not =j,\\
\zeta\left(  3\right)  +\zeta\left(  2\right)  H_{j-1}-H_{j-1}H_{j-1}^{\left(
2\right)  }-H_{j-1}^{\left(  3\right)  }, & s=j.
\end{array}
\right.
\]
Here we have used \cite[Lemma 1]{SC}
\[
\sum_{n=1}^{\infty}\frac{H_{n}}{\left(  n+j\right)  ^{2}}=\zeta\left(
3\right)  +\zeta\left(  2\right)  H_{j-1}-H_{j-1}H_{j-1}^{\left(  2\right)
}-H_{j-1}^{\left(  3\right)  }%
\]
and
\[
\sum_{n=1}^{\infty}\frac{H_{n}}{\left(  n+s\right)  \left(  n+j\right)
}=\frac{1}{2\left(  j-s\right)  }\left(  \left(  H_{j-1}\right)  ^{2}%
+H_{j-1}^{\left(  2\right)  }-\left(  H_{s-1}\right)  ^{2}-H_{s-1}^{\left(
2\right)  }\right)  ,
\]
which is a consequence of \cite[Eq.(2.30)]{XZZ}
\[
\sum_{n=1}^{\infty}\frac{H_{n}}{\left(  n+s\right)  \left(  n+j\right)
}=\frac{1}{\left(  j-s\right)  }\left\{  \sum_{k=1}^{j-1}\frac{H_{k}}{k}%
-\sum_{k=1}^{s-1}\frac{H_{k}}{k}\right\}  ,\text{ }j>s,
\]
and \cite[Lemma 1.1]{XL}
\[
\sum_{k=1}^{n}\frac{H_{k}}{k}=\frac{\left(  H_{n}\right)  ^{2}+H_{n}^{\left(
2\right)  }}{2}.
\]
Hence, in the light of (\ref{ES}) and (\ref{Zfpj0}) with $f_{n}=H_{n}$ (or
\cite[p. 322]{SC}), the series on the LHS of (\ref{9}) can be written as a
finite combination of the Riemann zeta values.

The evaluation of the second series on the RHS of (\ref{7}) in terms of the
Riemann zeta values and the linear Euler sums follows from the following
equations \cite[Eqs. (4.7)]{XZZ}%
\begin{align*}
\sum_{n=1}^{\infty}\frac{\left(  H_{n}\right)  ^{2}}{n^{p}\binom{n+l}{l}}  &
=\sum_{a=1}^{l}\sum_{m=1}^{p-1}\left(  -1\right)  ^{a+m}a^{m-1}\binom{l}%
{a}\sum_{n=1}^{\infty}\frac{\left(  H_{n}\right)  ^{2}}{n^{p+1-m}}\\
&  +\sum_{a=1}^{l}\left(  -1\right)  ^{a+p}a^{2-p}\binom{l}{a}\sum
_{n=1}^{\infty}\frac{\left(  H_{n}\right)  ^{2}}{n\left(  n+a\right)  },
\end{align*}
and \cite[Eq. (2)]{BBG}%
\begin{align*}
\sum_{n=1}^{\infty}\frac{\left(  H_{n}\right)  ^{2}}{n^{p}}  &  =\zeta
_{H^{\left(  2\right)  }}\left(  p\right)  +\frac{\left(  p^{2}+p-3\right)
}{3}\zeta\left(  p+2\right)  +\zeta\left(  2\right)  \zeta\left(  p\right) \\
&  -\frac{p}{2}\sum_{k=0}^{p-2}\zeta\left(  p-k\right)  \zeta\left(
k+2\right)  +\frac{1}{3}\sum_{k=0}^{p-2}\zeta\left(  p-k\right)  \sum
_{j=1}^{k-1}\zeta\left(  j+1\right)  \zeta\left(  k+1-j\right)
\end{align*}
and \cite[(2.39)]{XZZ} \
\begin{align*}
\sum_{n=1}^{\infty}\frac{\left(  H_{n}\right)  ^{2}}{n\left(  n+a\right)  }
&  =\frac{3\zeta\left(  3\right)  }{a}+\frac{\left(  H_{a}\right)  ^{3}%
+3H_{a}H_{a}^{\left(  2\right)  }+2H_{a}^{\left(  3\right)  }}{3a}\\
&  -\frac{\left(  H_{a}\right)  ^{2}+H_{a}^{\left(  2\right)  }}{a^{2}}%
-\frac{1}{a}\sum_{i=1}^{a-1}\frac{H_{i}}{i^{2}}+\frac{\zeta\left(  2\right)
H_{a-1}}{a}.
\end{align*}
The evaluation of the third series in (\ref{7}) is already shown in
(\ref{L3}). Hence the proof is completed.
\end{proof}

Now with a similar approach, we consider the evaluation of the second series
on the RHS of\ (\ref{hnrhnq}).

\begin{proposition}
\label{teo1}Let $l,p,r$ be non-negative integers with $p+l>r$. Then the series%
\[
\sum_{n=1}^{\infty}\frac{h_{n}^{\left(  r+1\right)  }}{n^{p}\left(
n+m\right)  \binom{n+l}{l}}%
\]
can be written as a finite combination of the Riemann zeta values and harmonic numbers.
\end{proposition}

\begin{proof}
One can see from (\ref{6}) that%
\begin{align}
\sum_{n=1}^{\infty}\frac{h_{n}^{\left(  r+1\right)  }}{n^{p}\left(
n+m\right)  \binom{n+l}{l}}  &  =\frac{1}{r!}\sum_{k=0}^{r}%
\genfrac{[}{]}{0pt}{}{r+1}{k+1}%
\left\{  \sum_{j=1}^{r}\sum_{n=1}^{\infty}\frac{1}{n^{p-k}\left(  n+m\right)
\binom{n+l}{n}\left(  n+j\right)  }\right. \nonumber\\
&  \left.  +\sum_{n=1}^{\infty}\frac{H_{n}}{n^{p-k}\left(  n+m\right)
\binom{n+l}{n}}-\sum_{n=1}^{\infty}\frac{H_{r}}{n^{p-k}\left(  n+m\right)
\binom{n+l}{n}}\right\}  . \label{17}%
\end{align}
The second and third series on the RHS of (\ref{17}) are known from (\ref{9})
and (\ref{Hplj}), respectively. Therefore, we only need to consider the first
series. Note that when $m\neq j$ the series
\[
\sum_{n=1}^{\infty}\frac{1}{n^{p}\left(  n+m\right)  \left(  n+j\right)
\binom{n+l}{l}}%
\]
can be evaluated from (\ref{Hplj}) by writing it as
\[
\frac{1}{j-m}\left\{  \sum_{n=1}^{\infty}\frac{1}{n^{p}\left(  n+m\right)
\binom{n+l}{l}}-\sum_{n=1}^{\infty}\frac{1}{n^{p}\left(  n+j\right)
\binom{n+l}{l}}\right\}  .
\]
When $m=j$, we have from (\ref{13}) that%
\[
\sum_{n=1}^{\infty}\frac{1}{n^{p}\left(  n+m\right)  ^{2}\binom{n+l}{l}}%
=\sum_{s=1}^{l}\left(  -1\right)  ^{s-1}\binom{l}{s}s\sum_{n=1}^{\infty}%
\frac{1}{n^{p}\left(  n+m\right)  ^{2}\left(  n+s\right)  }.
\]
It is an easy matter to derive that
\[
\sum_{n=1}^{\infty}\frac{1}{n^{p}\left(  n+m\right)  ^{2}\left(  n+s\right)
}=\frac{1}{s}\sum_{n=1}^{\infty}\frac{1}{n^{p}\left(  n+m\right)  ^{2}}%
-\frac{1}{s}\sum_{n=1}^{\infty}\frac{1}{n^{p-1}\left(  n+m\right)  ^{2}\left(
n+s\right)  }.
\]
This reduction formula yields to
\begin{align}
&  \sum_{n=1}^{\infty}\frac{1}{n^{p}\left(  n+m\right)  ^{2}\left(
n+s\right)  }\nonumber\\
&  =\sum_{v=1}^{p-1}\frac{\left(  -1\right)  ^{p-v+1}}{s^{p-v}}\sum
_{n=1}^{\infty}\frac{1}{n^{v+1}\left(  n+m\right)  ^{2}}+\frac{\left(
-1\right)  ^{p}}{s^{p}}\sum_{n=1}^{\infty}\frac{1}{\left(  n+m\right)
^{2}\left(  n+s\right)  }. \label{4}%
\end{align}
The first series on the RHS of (\ref{4}) is nothing but (\ref{3}) with $s=v+1$
and $t=2$. Besides, the second series is
\[
\sum_{n=1}^{\infty}\frac{1}{\left(  n+m\right)  ^{3}}=\zeta\left(  3\right)
-H_{m}^{\left(  3\right)  },\text{ if }m=s,
\]
and from (\ref{3})%
\[
\sum_{n=1}^{\infty}\frac{1}{\left(  n+s\right)  \left(  n+m\right)  ^{2}%
}=\frac{\zeta\left(  2\right)  }{s-m}+\frac{H_{m}-H_{s}}{\left(  s-m\right)
^{2}}-\frac{H_{m}^{\left(  2\right)  }}{s-m},\text{ if }m\neq s.
\]
Combining the results above gives%
\begin{align*}
&  \sum_{n=1}^{\infty}\frac{1}{n^{p}\left(  n+m\right)  ^{2}\left(
n+s\right)  }\\
&  =\sum_{v=0}^{p-1}\frac{\left(  -1\right)  ^{p}}{s^{p-v}m^{v+1}}\left\{
\zeta\left(  2\right)  -H_{m}^{\left(  2\right)  }-\left(  v+1\right)
\frac{H_{m}}{m}+\sum_{n=0}^{v-1}\left(  -m\right)  ^{n}\left(  v-n\right)
\zeta\left(  n+2\right)  \right\} \\
&  +\frac{\left(  -1\right)  ^{p}}{s^{p}}B_{3}\left(  m,s\right)  ,
\end{align*}
where
\[
B_{3}\left(  m,s\right)  =\left\{
\begin{array}
[c]{lc}%
\zeta\left(  3\right)  -H_{m}^{\left(  3\right)  }, & m=s,\\
\frac{\zeta\left(  2\right)  }{s-m}+\frac{H_{m}-H_{s}}{\left(  s-m\right)
^{2}}-\frac{H_{m}^{\left(  2\right)  }}{s-m}, & m\not =s.
\end{array}
\right.
\]
The proof is then completed.
\end{proof}

Thus in the light of Lemma \ref{lemL10}, Proposition \ref{pro1} and
Proposition \ref{teo1}, we reach at the proof of Theorem \ref{mteo2}.

\section{Further consequences}

In this section, we present the connection of the series $\sum\limits_{n=1}%
^{\infty}h_{n}^{\left(  r+1\right)  }/n^{p}\binom{n+l}{l}$ with some results
in the literature, for instance with the shifted Euler sums and the Hurwitz
zeta function.

In \cite[p.364]{MD} Euler's sum was expressed in terms of a series involving
the Hurwitz zeta function $\zeta\left(  p,k\right)  $ as%
\begin{equation}
\sum_{k=1}^{\infty}\frac{\zeta\left(  p,k\right)  }{k}=\zeta_{H}\left(
p\right)  =\frac{1}{2}\left(  p+2\right)  \zeta\left(  p+1\right)  -\frac
{1}{2}\sum_{n=1}^{p-2}\zeta\left(  m-n\right)  \zeta\left(  n+1\right)
,\text{ }p\in%
\mathbb{N}
\backslash\left\{  1\right\}  . \label{HZS}%
\end{equation}
On the other hand, Xu and Li in \cite[Theorem 2.1]{XL} considered the
following shifted form of Euler's sum%
\begin{equation}
\sum_{n=r+1}^{\infty}\frac{H_{n}}{\left(  n-r\right)  ^{p}}=\zeta_{H}\left(
p\right)  -\sum_{m=1}^{p-1}\left(  -1\right)  ^{m}\zeta\left(  p+1-m\right)
H_{r}^{\left(  m\right)  }-\left(  -1\right)  ^{p}\sum_{m=1}^{r}\frac{H_{m}%
}{m^{p}}. \label{xl}%
\end{equation}
Surprisingly, we observe that the series involving hyperharmonic numbers and
reciprocal binomial coefficients correspond to the shifted forms of the series
involving the Hurwitz zeta function and Euler's sum. These correspondences,
follow by utilizing the representations
\begin{equation}
\binom{n+r}{r}\sum_{k=1}^{n}\frac{1}{r+k}=h_{n}^{\left(  r+1\right)  }%
=\binom{n+r}{r}\left(  H_{n+r}-H_{r}\right)  , \label{19}%
\end{equation}
in $\sum\limits_{n=1}^{\infty}h_{n}^{\left(  r+1\right)  }/n^{p}\binom{n+r}%
{r},$\ respectively, give rise to the following result.

\begin{corollary}
For positive integers $p$ and $r$ with $p>1$, we have%
\[
\sum_{k=1}^{\infty}\frac{\zeta\left(  p,k\right)  }{r+k}=\zeta_{H}\left(
p\right)  +\sum_{j=2}^{p-1}\left(  -1\right)  ^{j-1}\zeta\left(  p+1-j\right)
H_{r}^{\left(  j\right)  }+\left(  -1\right)  ^{p-1}\sum_{j=1}^{r}\frac{H_{j}%
}{j^{p}}.
\]

\end{corollary}

The following results are binomial extensions of (\ref{xl}).

\begin{corollary}
For $q\in\mathbb{N}$ and non-negative integers $p,r$ with $p+q>1$, we have%
\begin{align*}
&  \sum_{n=r+1}^{\infty}\frac{H_{n}}{\left(  n-r\right)  ^{p}\binom{n+q}{q}%
}=\binom{r+q}{q}^{-1}\sum_{n=1}^{\infty}\frac{h_{n}^{\left(  r+1\right)  }%
}{n^{p}\binom{n+q+r}{n}}\\
&  \qquad\qquad-H_{r}\sum_{a=1}^{q}\left(  -1\right)  ^{a}\binom{q}%
{a}a\left\{  \sum_{m=1}^{p-1}\frac{\left(  -1\right)  ^{m}}{\left(
r+a\right)  ^{m}}\zeta\left(  p+1-m\right)  +\frac{\left(  -1\right)  ^{p}%
}{\left(  r+a\right)  ^{p}}H_{r+a}\right\}  .
\end{align*}

\end{corollary}

\begin{proof}
Let $l>r$. From (\ref{19}), we have%
\[
\sum_{n=1}^{\infty}\frac{h_{n}^{\left(  r+1\right)  }}{n^{p}\binom{n+l}{l}%
}=\frac{l!}{r!}\sum_{n=1}^{\infty}\frac{H_{n+r}}{n^{p}\left(  n+r+1\right)
\cdots\left(  n+l\right)  }-H_{r}\frac{l!}{r!}\sum_{n=1}^{\infty}\frac
{1}{n^{p}\left(  n+r+1\right)  \cdots\left(  n+l\right)  },
\]
that is,%
\[
\sum_{n=1}^{\infty}\frac{H_{n+r}}{n^{p}\binom{n+l}{l-r}}=\binom{l}{r}^{-1}%
\sum_{n=1}^{\infty}\frac{h_{n}^{\left(  r+1\right)  }}{n^{p}\binom{n+l}{l}%
}+H_{r}\left(  l-r\right)  !\sum_{n=1}^{\infty}\frac{1}{n^{p}\left(
n+r+1\right)  \cdots\left(  n+l\right)  }.
\]
The first series on the RHS is already given in (\ref{L10}) (with the use of
(\ref{L3}), (\ref{Hplj}) and (\ref{L9})). The second can be written as%
\begin{align*}
&  \sum_{n=1}^{\infty}\frac{1}{n^{p}\left(  n+r+1\right)  \cdots\left(
n+l\right)  }\\
&  =\frac{1}{\left(  l-r-1\right)  !}\sum_{a=0}^{l-r-1}\left(  -1\right)
^{a}\binom{l-r-1}{a}\sum_{n=1}^{\infty}\frac{1}{n^{p}\left(  n+r+1+a\right)
},
\end{align*}
by (\ref{13}). Thus, writing $q+r$ for $l$ and using (\ref{Hp0j}) complete the proof.
\end{proof}

In a similar way, noting
\[
\left(  n+l+1\right)  \cdots\left(  n+r\right)  =\sum_{k=0}^{r-l}%
\genfrac{[}{]}{0pt}{}{r-l}{k}%
_{l+1}n^{k}%
\]
from (\ref{1s}), we state the following.

\begin{corollary}
For $p,q\in\mathbb{N}$ and non-negative integer $l$ with $p>q+1$, we have
\[
\sum_{n=q+l+1}^{\infty}\frac{\binom{n}{q}H_{n}}{\left(  n-l-q\right)  ^{p}%
}=\binom{q+l}{l}\sum_{n=1}^{\infty}\frac{h_{n}^{\left(  q+l+1\right)  }}%
{n^{p}\binom{n+l}{l}}+\frac{H_{q+l}}{q!}\sum_{k=0}^{q}%
\genfrac{[}{]}{0pt}{}{q}{k}%
_{l+1}\zeta\left(  p-k\right)  .
\]

\end{corollary}

As a final note, we would like to emphasize that it is possible to evaluate
different nonlinear Euler-type sums by particular choices of $f_{n}$ such as
$\left(  H_{n}\right)  ^{2}$ and $H_{n}^{\left(  r\right)  }H_{n}^{\left(
q\right)  }$ in \textbf{(}\ref{Zfpjq}) together with the results in \cite{XZZ}
and \cite{BBG}.

\end{document}